\documentclass[review]{elsarticle}

\usepackage{lineno,hyperref}
\usepackage{xcolor}
\modulolinenumbers[5]
\usepackage{amssymb}
\usepackage{amsmath}
\usepackage{amsthm}

  \usepackage{lineno}
  \newtheorem{thm}{Theorem}
\newtheorem{lem}{Lemma}
\newtheorem{cor}{Corollary}

\begin{document}

\begin{frontmatter}

\title{Positive semidefinite interval of matrix pencil and its applications for the generalized trust region subproblems}

\author[mymainaddress]{Thi Ngan Nguyen}
\ead[url]{nguyenthingan@ttn.edu.vn}
\author[mymainaddress]{Van-Bong Nguyen\corref{mycorrespondingauthor}}
\cortext[mycorrespondingauthor]{Corresponding author}
\ead{nvbong@ttn.edu.vn}

\address[mymainaddress]{Department of Mathematics, Tay Nguyen
University, 632090, Vietnam}

\begin{abstract}
We are concerned with finding the set $I_{\succeq}(A,B)$ of real values $\mu$ such that the matrix pencil $A+\mu B$ is positive semidefinite.
If $A, B$ are not simultaneously diagonalizable via congruence (SDC), $I_{\succeq}(A,B)$   either is empty or has only
one value $\mu.$   When $A, B$ are SDC,
  $I_{\succeq}(A,B),$ if not empty, can be a singleton or an interval. Especially, if $I_{\succeq}(A,B)$ is an interval and
   at least one of the matrices is nonsingular then its interior is the positive definite interval
$I_{\succ}(A,B).$ If $A, B$  are both singular, then even $I_{\succeq}(A,B)$ is an interval, its interior
may not be $I_{\succ}(A,B),$ but $A, B$ are then decomposed to block diagonals of submatrices
$A_1, B_1$ with $B_1$ nonsingular  such that $I_{\succeq}(A,B)=I_{\succeq}(A_1,B_1).$
Applying $I_{\succeq}(A,B),$  the {\em hard-case} of the generalized trust-region subproblem (GTRS)
can be dealt with by only solving a system of linear equations or reduced to the {\em easy-case}
of a GTRS of smaller size.

\end{abstract}

\begin{keyword}
Simultaneously diagonalizable via congruence \sep  Trust region subproblem \sep
 Generalized trust region subproblem\sep Positive semidefinite interval \sep  Matrix pencil
\MSC[2010] 15A18\sep 15A21\sep 15A22 \sep 15A23\sep 15A27 \sep 90C20\sep  90C26
\end{keyword}

\end{frontmatter}

\linenumbers

\section{Introduction}
	Let $\mathcal{S}^n$ be the space of $n\times n$ real symmetric matrices and  $A, B\in \mathcal{S}^n. $  In this paper we compute the set
\begin{align*}
I_{\succeq}(A,B)=\{\mu\in\Bbb R: A+\mu B\succeq0\}
\end{align*}
and then  apply it for solving the following  optimization problem
\begin{equation}\label{original}
 \hspace*{0.6cm}
\begin{array}{lll}
\lambda^*=&\min &f(x)=x^TAx+2a^Tx\\
&{\rm s.t.} & g(x)=x^TBx+2b^Tx+ c\le 0,
\end{array}
\end{equation}
where   $a, b\in \Bbb R^{n}$ are $n-$ dimension vectors,  $c\in\Bbb R.$ This problem is referred as the generalized trust region subproblem (GTRS) since it contains the
trust region subproblem (TRS), when $B=I$ is the identity matrix, $b=0$ and $c=-1,$ as a special case.

Our   study is inspired
by the following results obtained by Mor\'{e} \cite{MJ}: i)  a vector $x^*\in\Bbb R^n$ is a global optimizer of the
GTRS \eqref{original} if and only if $g(x^*)\le 0$ and there exists $\mu^*\in I_{\succeq}(A,B)\cap [0,\infty)$
such that
\begin{align}
\nabla f(x^*)+\mu^*\nabla g(x^*)=0,\label{mu1}\\
\mu^*g(x^*)=0 \label{mu2};
\end{align}
and ii) if the set defined by $I_{\succ}(A,B)=\{\mu\in\Bbb R: A+\mu B\succ0\}$ is nonempty then
it is an open interval  and
the function $\varphi(\mu):=g[x(\mu)]$ is   strictly decreasing on $I_{\succ}(A,B),$ unless $x(\mu)$
 uniquely solved from \eqref{mu1} is  constant on $ I_{\succ}(A,B).$
 These two results suggest that $\mu^*$ can be found efficiently whenever   $I_{\succeq}(A,B)$ is computed.
 Unfortunately, in literature,  results on $I_{\succeq}(A,B) $ are
  very limited. The earliest study on  computing  $I_{\succeq}(A,B)$  was, perhaps, by
  Caron and  Gould  \cite{Caron}, where the authors suppose that $A$ is positive semidefinite and $B$ is of rank one or two.
Song \cite{Song} computed  $I_{\succeq}(A,B)$ under the assumption that $\mathcal{R}(B)\subset\mathcal{R}(A),$
where $\mathcal{R}(\cdot)$ denotes the range of a matrix.  Mor\'{e} \cite{MJ}, Adachi et al.   \cite{Adachi}
computed $I_{\succeq}(A,B)$ under the assumption that $I_{\succ}(A,B)\ne\emptyset.$
Out of above conditions, computing $I_{\succeq}(A,B)$
is still an open question so far. However, if  $I_{\succ}(A,B)\ne \emptyset,$ $I_{\succeq}(A,B) $ is then the closure of
$I_{\succ}(A,B) $ and
computing $I_{\succeq}(A,B)$ is much easier since the matrices $A,B$ are then  simultaneously diagonalizable via congruence (SDC), i.e.,
there exists a nonsingular matrix $P$ such that $P^TAP$ and $P^TBP$ are all diagonal, please see   \cite{Horn}.
    Mor\'{e} \cite{MJ}    proposed an algorithm for finding $\mu^*$ under the assumption that $I_{\succ}(A,B)$
is nonempty, while the case $I_{\succ}(A,B)=\emptyset$ is referred as the {\em hard-case}.
The assumption $I_{\succ}(A,B)\ne\emptyset$ is so strictive that it restricts the GTRS \eqref{original}  into a
very special class which always attains a unique optimal solution. Moreover, if
$I_{\succ}(A,B)\ne\emptyset$  then  $A, B$ are SDC
and the GTRS \eqref{original} is  equivalently reformulated as a  convex second-order cone programming (SOCP) \cite{Ben-Tal}.
This result is then
   extended that any bounded GTRS can be reformulated as an SOCP even the SDC condition fails  \cite{Jiang1}.
 The assumption  $I_{\succ}(A,B)\ne\emptyset$  also allows
to solve GTRS \eqref{original} by solving only one eigenpair of a generalized eigenvalue problem of $(2n+1)\times(2n+1)$ dimension \cite{Adachi}.
A more recent result by  Jiang and  Li  can solve GTRS \eqref{original} even in linear time $O(n)$ but under a stricter condition that
 there exist $\mu\in(0,1]$ and $\lambda_{\rm min}(B)\le -\xi<0$ such that $\mu A+(1-\mu)B\succeq \xi I,$
 here $\lambda_{\rm min}(B)$ is the smallest eigenvalue of $B$ \cite{Jiang2}.  Hsia et al. \cite{Hsia}
 solve the GTRS \eqref{original} when $A, B$ are SDC, while the case when $A, B$ are not SDC is still referred as the {\em hard-case}.

In this paper, we show that not only when $I_{\succ}(A,B)\ne \emptyset$ but also even $I_{\succ}(A,B)=\emptyset$
     the interval $I_{\succeq}(A,B)$ can be computed by only solving  a generalized eigenvalue problem of $n\times n$ dimension. 
     Specifically, we show that   if $A, B$ are not SDC then $I_{\succeq}(A,B)$
either  is  empty or has only one point: $I_{\succeq}(A,B)=\{\mu\}.$ Such a value $\mu$ can be found efficiently
and then checked whether $\mu^*=\mu.$
  If $A, B$ are SDC
and $B$ is nonsingular, then  the matrix pencil
$A+\mu B$ can always be decomposed into the form
\begin{align*}
 P^T(A+\mu B)P=\texttt{diag}((\lambda_1+\mu)B_1,(\lambda_2+\mu)B_2\ldots,(\lambda_k+\mu)B_k),
\end{align*}
where $B_i$ are $m_i\times m_i$ symmetric matrices, $\lambda_1>\lambda_2>\ldots>\lambda_k$ are distinct
eigenvalues of the matrix $B^{-1}A.$  The set
$I_{\succeq}(A,B)$ is then computed quickly since $A+\mu B\succeq0$ is equivalent to  $(\lambda_i+\mu)B_i\succeq0$ for
all $i=1,2,\ldots,k.$ If $B$ is singular and  $A$ is nonsingular, we can decompose
$B, A$ to the   taking the forms $U^TBU=\texttt{diag}(B_1,0),$
$U^TAU=\texttt{diag}(A_1,A_3),$
where  $B_1, A_1$ are symmetric of the same size, $B_1$ is nonsingular such that
  if $A_3\succ0$ then $I_{\succeq}(A,B)=I_{\succeq}(A_1,B_1),$ otherwise
 $I_{\succeq}(A,B)= \emptyset.$
Especially,  if $I_{\succeq}(A,B)$ has more than one point, then $I_{\succ}(A,B)\ne \emptyset$ and $I_{\succ}(A,B)={\rm int}(I_{\succeq}(A,B)),$
 please see Corollary \ref{cor1} below.
If $A, B$ are SDC  and both are singular, then there always exists a nonsingular
matrix $U$ such that
$A, B$ are decomposed  to either the following form
\begin{align}\label{l1}
 U^TBU=\texttt{diag}(B_1, 0) \text{ and }
 U^TAU=\texttt{diag}(A_1, 0)
\end{align}   or
\begin{align}\label{l2}
 U^TBU=\texttt{diag}(B_1, 0) \text{ and }
 U^TAU=\texttt{diag}(A_1, A_4),
\end{align}
where $A_4$ is diagonal and, in both cases, $A_1, B_1$ are diagonal and of the same size, $B_1$ is nonsingular.
If $A, B$ are decomposed to \eqref{l2} and $A_4$  has even one  diagonal negative element then
$I_{\succeq}(A,B)= \emptyset.$
Otherwise, in both cases, $I_{\succeq}(A,B)=I_{\succeq}(A_1,B_1)$ with $B_1$ nonsingular.
To apply $I_{\succeq}(A,B)$ for solving   GTRS \eqref{original}, we consider the set of candidates of Lagrange multipliers
$I=I_{\succeq}(A,B)\cap[0,\infty).$
\begin{enumerate}
\item If $I=\emptyset,$ \eqref{original} has no optimal solution since it is unbounded from below;
\item If $I$ is singleton: $I=\{\mu\},$ we  need only to solve  linear equations to check whether $\mu^*=\mu;$
\item If $I$ is an interval and $I_{\succ}(A,B)\ne \emptyset,$ then there is a unique optimal Lagrange multiplier $\mu^* \in I.$ If
$\mu^*$ is not an endpoint of $I,$ a bisection algorithm can find $\mu^*$ in the interior of $I$
  since the function $\varphi(\mu)$
is  strictly decreasing on $I.$
 \item If $I$ is an interval and $I_{\succ}(A,B)= \emptyset,$ $A, B$ are converted to the form either \eqref{l1} or \eqref{l2}
and $I_{\succ}(A_1,B_1)\ne \emptyset.$ In this case, the  GTRS \eqref{original} is either unbounded below or reduced
to a GTRS of $p $ variables with matrices  $A_1, B_1$ such that $I_{\succ}(A_1,B_1)\ne\emptyset.$
\end{enumerate}

\section{Computing the positive semidefinite interval $I_{\succeq}(A,B)$}\label{sec1}
In this section, we show computing $I_{\succeq}(A,B)$ in two separate cases: $A, B$ are SDC and $A, B$ are not SDC.
For the former case, we need first the following result.
 \begin{lem}[\cite{Greub}]\label{3}
Let  $A, B\in \mathcal{S}^n $  and  $B$ be  nonsingular.
 Then $A, B$ are SDC  if and only if  there is a  nonsingular real matrix $P$ such that $P^{-1}B^{-1}AP$ is  a real diagonal matrix.
 \end{lem}

Now, if  $A, B$ are SDC and $B$ is   nonsingular, by Lemma \ref{3}, there is a nonsingular matrix $P$ such that
$$ J:= P^{-1} B^{-1}A P =\texttt{diag}(\lambda_1I_{m_1}, \ldots, \lambda_k I_{m_k}),$$
   is a diagonal matrix, where $\lambda_1, \lambda_2, \ldots, \lambda_k$
   are the $k$ distinct eigenvalues of $B^{-1}A,$ $I_{m_t}$ is the identity matrix of size $m_t \times m_t$
   and $m_1+m_2+\ldots+m_k=n.$
  We can suppose without loss of generality that
      $\lambda_1>\lambda_2>\ldots>\lambda_k.$
Using $J$ together with the following result we show how to simultaneously decompose $A$ and $B$ into block diagonals.
\begin{lem}[\cite{Uhlig76}]\label{bd1}
Let $K $ be a Jordan matrix of form
$$K=\texttt{diag}(C(\lambda_{1}), C(\lambda_{2}), \cdots, C(\lambda_{k})),$$
where $C(\lambda_{i})=\texttt{diag}(K_{i_{1}}(\lambda_{i}), K_{i_{2}}(\lambda_{i}),\cdots,K_{i_{t_i}}(\lambda_{i})), i=1,2,\ldots, k, $
are Jordan blocks associated with eigenvalue  $\lambda_{i}$ and
\par
\[
K_{i_j}(\lambda_i)=
\left(
\begin{array}{cccccc}
 \lambda_{i} & 1 & 0 &\cdots &\cdots & 0\\
0 & \lambda_{i} & 1&\cdots &\cdots & 0\\
 \cdots &\cdots &\cdots&\cdots &\cdots &\cdots\\
\cdots &\cdots &\cdots&\cdots &\cdots &\cdots\\
 0 & 0 &\cdots &\cdots &\lambda_{i} & 1\\
 0 & 0 &\cdots &\cdots & 0 &\lambda_{i}
\end{array}
\right)_{(ij)}, \hspace{1cm}  j=1, 2, \cdots, t_i.
\]
For a symmetric matrix $S,$ if  $SK$ is symmetric, then $S$ is block diagonal  and  $S=\texttt{diag}(S_{1}, S_{2},\cdots, S_{k})$ with
$${\rm dim} S_{i}={\rm dim }C(\lambda_{i}).$$
\end{lem}

Observe  that $P^TBP.J=P^TAP$ and $P^TAP$ is symmetric.  Lemma \ref{bd1} indicates that  $P^TBP$ is a block diagonal matrix with the same partition as $J.$ That is
\begin{align}\label{ct1}
P^TBP=\texttt{diag}(B_1,B_2\ldots,B_k),
\end{align}
where $B_t$ is real symmetric matrices of size $m_t \times m_t$ for every
$t=1,2, \ldots, k.$
We now have
\begin{align}\label{ct2}
P^TAP=P^TBP.J=\texttt{diag}(\lambda_1B_1,\lambda_2B_2\ldots,\lambda_kB_k).
\end{align}
Both \eqref{ct1} and \eqref{ct2} show that   $A, B$ are now decomposed into the same block structure and the matrix
pencil $A+\mu B$ now becomes
\begin{align}\label{ct3} P^T(A+\mu B)P=\texttt{diag}((\lambda_1+\mu)B_1,(\lambda_2+\mu)B_2\ldots,(\lambda_k+\mu)B_k).
\end{align}
The requirement $A+\mu B \succeq 0$ is then equivalent to
\begin{equation}\label{pt1}
\begin{split}
 (\lambda_i+\mu)B_i \succeq 0,   i=1,2,\ldots,k.
\end{split}
\end{equation}
Using \eqref{pt1} we compute $I_{\succeq}(A,B)$ as follows.
 \begin{thm}\label{dl1}
Suppose $A, B\in \mathcal{S}^n$ are  SDC and $B$ is  nonsingular.
\begin{enumerate}
 \item If $B \succ 0$ then  $I_{\succeq}(A,B)=[-\lambda_k, + \infty);$
 \item If $B \prec 0$ then  $I_{\succeq}(A,B)=(- \infty, -\lambda_1];$
\item If $B$ is indefinite  then
\begin{itemize}
\item[(i)] if  $B_1, B_2,\ldots,B_{t} \succ 0$  and $B_{t+1}, B_{t+2},\ldots, B_k \prec 0$ for some  $t \in \{1,2,\ldots,k\},$ then $I_{\succeq}(A,B)=[ -\lambda_t,-\lambda_{t+1}].$
\item [(ii)]  if  $B_1, B_2,\ldots, B_{t-1} \succ 0,$ $B_t$ is indefinite and $B_{t+1}, B_{t+2},\ldots, B_k \prec 0,$ then $I_{\succeq}(A,B)=\{-\lambda_t\},$
\item [(iii)] in other cases, that is either $B_i, B_j$ are  indefinite for some $i\ne j$  or $B_i \prec 0, B_j \succ 0$ for some $i<j$ or
$B_i $ is indefinite and $B_j \succ 0$ for some $i<j,$  then $I_{\succeq}(A,B)=\emptyset.$
\end{itemize}
 \end{enumerate}
\end{thm}
\begin{proof}

\begin{enumerate}
\item  If $B \succ 0$ then $B_i \succ 0 ~ \forall i=1,2,\ldots,k.$ The inequality \eqref{pt1} is then equivalent to
$\lambda_i+\mu \geq 0~ \forall i=1,2,\ldots, k.$
Since $\lambda_1>\lambda_2>\ldots>\lambda_k,$ we need only $\mu \geq -\lambda_k.$
This shows $I_{\succeq}(A,B)=[-\lambda_k, + \infty).$
\item Similarly, if  $B \prec 0$ then $B_i \prec 0~ \forall i=1,2,\ldots,k.$ The inequality \eqref{pt1} is then equivalent to
$\lambda_i+\mu \leq 0~ \forall i=1,2,\ldots, k.$ Then  $I_{\succeq}(A,B)=(- \infty, -\lambda_1].$
\item  The case $B$ is indefinite:
\begin{itemize}
\item [(i)] if $B_1, B_2,\ldots, B_{t} \succ 0$  and $B_{t+1}, B_{t+2},\ldots, B_k \prec 0$ for some  $t \in \{1,2,\ldots,k\},$
the inequality \eqref{pt1} then implies
$$ \begin{cases}  \lambda_i+\mu \geq 0, \forall i=1,2,\ldots, t,\\
\lambda_i+\mu \leq 0, \forall i=t+1,\ldots,k.
\end{cases}
$$
 Since $\lambda_1>\lambda_2>\ldots>\lambda_k,$  we have  $I_{\succeq}(A,B)=[ -\lambda_t,-\lambda_{t+1}].$
\item[(ii)] if  $B_1, B_2,\ldots, B_{t-1} \succ 0,$ $B_t$ is indefinite and $B_{t+1},B_{t+2},\ldots, B_k \prec 0$ for some
 $t \in \{1,2,\ldots,k\}.$   The inequality \eqref{pt1} then implies
$$ \begin{cases}  \lambda_i+\mu \geq 0, \forall i=1,2,\ldots, t-1\\
\lambda_t+\mu = 0\\
\lambda_i+\mu \leq 0, \forall i=t+1,\ldots,k.
\end{cases}
$$
 Since $\lambda_1>\lambda_2>\ldots>\lambda_k,$  we have  $I_{\succeq}(A,B)=\{-\lambda_t\}.$
\item [(iii)] if $B_i, B_j$ are  indefinite, \eqref{pt1} implies $\lambda_i+\mu=0$ and $\lambda_j+\mu=0.$ This cannot happen since
$\lambda_i\ne \lambda_j.$  If  $B_i \prec 0$ and $ B_j \succ 0$ for some $i<j,$ then
$$ \begin{cases}
 \lambda_i+\mu \leq 0\\
\lambda_j+\mu \geq 0
\end{cases}
$$
implying $-\lambda_j \leq \mu \leq -\lambda_i.$ This also cannot happen since
  $\lambda_i>\lambda_j.$  Finally, if  $B_i $ is indefinite and $B_j \succ 0$ for some $i<j.$ Again, by \eqref{pt1},
$$ \begin{cases}\label{pt2} \lambda_i+\mu = 0\\
\lambda_j+\mu \geq 0
\end{cases}
$$
implying $ \lambda_i \leq  \lambda_j.$ This also cannot happen.
So  $I_{\succeq}(A,B)=\emptyset$ in these all three cases.
\end{itemize}
\end{enumerate}
\end{proof}

The proof of Theorem \ref{dl1} indicates that if $A,B$ are SDC, $B$ is nonsingular and $I_{\succeq}(A,B)$ is an interval
then $I_{\succ}(A,B)$ is nonempty. In that case we have
  $I_{\succ}(A,B)={\rm int}(I_{\succeq}(A,B)),$ please see \cite{MJ}.
  If $B$ is singular and $A$ is nonsingular, we have the following result.

\begin{thm}\label{dl2}
Suppose $A, B\in \mathcal{S}^n$ are  SDC, $B$ is singular and $A$ is  nonsingular.  Then
\begin{enumerate}
\item[(i)] there always exists a nonsingular matrix $U$ such that
$$U^TBU=\texttt{diag}(B_1,0),$$
$$U^TAU=\texttt{diag}(A_1,A_3),$$
where  $B_1, A_1$ are symmetric of the same size, $B_1$ is nonsingular;
\item[(ii)] if $A_3\succ0$ then $I_{\succeq}(A,B)=I_{\succeq}(A_1,B_1).$ Otherwise,
 $I_{\succeq}(A,B)= \emptyset.$
 \end{enumerate}
\end{thm}
\begin{proof} (i) Since $B$ is symmetric and singular, there is an orthogonal matrix $Q_1$ that puts $B$ into the form
$$ \hat{B} =Q_1^TBQ_1=\texttt{diag}(B_1, 0)$$
such that $B_1$ is a nonsingular symmetric matrix of size $p\times p,$ where $p={\rm rank}(B).$
Let $\hat{A} :=Q_1^TAQ_1.$ Since $A, B$ are SDC,  $\hat{A}, \hat{B}$ are  SDC too (the converse also holds true).
We can write $\hat{A}$ in the following form
\begin{align}\label{b}
\hat{A}=Q_1^TAQ_1=\left(\begin{matrix}M_1&M_2\\M_2^T&M_3\end{matrix}\right)
\end{align}  such that
			$M_1$ is a symmetric matrix of  size $ p\times p,$ $M_2$ is a $p \times(n-p)$ matrix,
$M_3$ is symmetric of size $(n-p)\times(n-p) $ and, importantly,  $M_3\ne0.$ Indeed,
if  $M_3=0$ then $\hat{A}=Q_1^TAQ_1=\left(\begin{matrix}M_1&M_2\\ M_2^T&0\end{matrix}\right).$
Then we can choose a nonsingular matrix $H$ written in the same partition as $\hat{A}:$ $H=\left(
\begin{matrix}
	H_1&H_2\\ H_3&H_4
\end{matrix}
\right)$ such that both $H^T\hat{B}H, H^T\hat{A}H$ are diagonal and
$H^T\hat{B}H$   is of the form
 \begin{align*}
H^T\hat{B} H= \left(\begin{matrix}H_1^T B_1H_1&H_1^T B_1H_2\\
		H_2^T B_1H_1&H_2^T B_1H_2\\
		\end{matrix}\right)
= \left(\begin{matrix}H_1^T B_1H_1&0\\
		0&0\\
		\end{matrix}\right),
\end{align*}
where $H_1^T B_1H_1$ is nonsingular.  This implies
  $H_2=0.$ On the other hand,
 \begin{align*}
 H^T\hat{A} H=\left(\begin{matrix}H_1^T M_1H_1+H_3^TM_2^TH_1+H_1^TM_2H_3&H_1^TM_2H_4\\H_4^TM_2^TH_1&0\\
		\end{matrix}\right)
\end{align*}
is diagonal implying that $H_1^TM_2H_4=0,$ and so
\begin{align*}
H^T\hat{A} H=\left(\begin{matrix}H_1^T M_1H_1+H_3^TM_2^TH_1+H_1^TM_2H_3&0\\0&0\\
		\end{matrix}\right).
\end{align*}
This cannot happen since  $\hat{A}$ is nonsingular.

Let  $P$ be an orthogonal matrix such that   $P^TM_3P=\texttt{diag}(A_3,0_{q-r}),$
where $A_3$ is a nonsingular diagonal matrix of size $r\times r, r\leq q$ and $p+q=n,$ and set
			$U_1=\texttt{diag}(I_p, P).$ We then have
\begin{align}\label{c}
\tilde{A}: =U_1^T\hat{A}U_1
=\left(
\begin{matrix}
	M_1&M_2P\\ (M_2P)^T& P^TM_3P
\end{matrix}
\right)
=\left(
\begin{matrix}
	M_1&A_4&A_5\\ A_4^T&A_3&0\\ A_5^T&0&0
			\end{matrix}
		\right),
 \end{align}
 where $\left(\begin{matrix}A_4 & A_5\end{matrix}\right)=M_2P,$  $A_4$ and $ A_5$ are of size  $p\times r$ and $p\times(q-r), r\le q,$ respectively.
 Let
\begin{align*}
U_2=\left(\begin{matrix}I_p&0&0\\-A_3^{-1}A_4^T&I_r&0\\0&0&I_{q-r}
			\end{matrix}\right) \text{  and  }  U=Q_1U_1U_2.
\end{align*}
We can verify that
\begin{align*}
U^TBU=U_2^TU_1^T (Q_1^TBQ_1)) U_1U_2=\hat B,
\end{align*}
and, by \eqref{c},
\begin{align*}
U^TAU=U_2^T\tilde A U_2=\left(
\begin{matrix}
	M_1-A_4A_3^{-1}A_4^T&0&A_5\\0&A_3&0\\A_5^T&0&0
\end{matrix}
\right).
\end{align*}
We denote $A_1:=M_1-A_4A_3^{-1}A_4^T$ and rewrite the matrices as follows
\begin{align*}
U^TBU=\texttt{diag}(B_1,0),  U^TAU=\left( \begin{matrix} A_1&0&A_5\\
0&A_3&0\\ A_5^T&0&0\end{matrix}\right).
\end{align*}
We now consider whether it can happen that $r<q.$
We note that $U^TAU, U^TBU$ are SDC.  We can choose a nonsingular congruence matrix $K$ written in the form
\begin{align*}
K=\left(
\begin{matrix}
	K_1&K_2&K_3\\ K_4&K_5&K_6\\ K_7&K_8&K_9
\end{matrix}
\right)
\end{align*}
such that not only the matrices  $K^TU^TAUK, K^TU^TBUK$ are diagonal but also
$ K^TU^TBUK$ is remained a $p\times p$ nonsingular submatrix at the northwest corner. That is
\begin{align*}
K^TU^T BU K=\left(\begin{matrix}K_1^T B_1K_1&K_1^T B_1K_2&K_1^T B_1K_3\\
		K_2^T B_1K_1&K_2^T B_1K_2&K_2^T B_1K_3\\K_3^TB_1K_1&K_3^TB_1K_2&K_3^TB_1K_3
		\end{matrix}\right)=\left(\begin{matrix}K_1^T B_1K_1&0&0\\
		0&0&0\\0&0&0
		\end{matrix}\right)
\end{align*}
is diagonal and  $K_1^T B_1K_1$ is nonsingular diagonal of size $p\times p.$  This
 implies that $K_2=K_3=0.$   
 Then
\begin{align*}
& K^TU^T AU K=\\
&=\left(\begin{matrix}\substack{K_1^T A_1K_1+K_1^T A_2K_7 \\ +K_4^TA_3K_4+K_7^TA_2^TK_1}&K_1^TA_2K_8+K_4^TA_3K_5&K_1^TA_2K_9+K_4^TA_3K_6\\K_8^TA_2^TK_1+K_5^TA_3^TK_4&K_5^TA_3K_5&K_5^TA_3K_6\\K_9^TA_2^TK_1+K_6^TA_3^TK_4&K_6^TA_3K_5&K_6^TA_3K_6
		\end{matrix}\right) \\
&= \left(\begin{matrix}K_1^T A_1K_1+K_1^T A_2K_7+K_4^TA_3K_4+K_7^TA_2^TK_1&0&0\\0&K_5^TA_3K_5&0\\0&0&K_6^TA_3K_6
		\end{matrix}\right)
\end{align*}
	 is diagonal implying that 
$$K_1^T A_1K_1+K_1^T A_2K_7+K_4^TA_3K_4+K_7^TA_2^TK_1, K_5^TA_3K_5,
K_6^TA_3K_6$$ are diagonal.
Note that  $U^TAU$ is nonsingular, $K_5^TA_3K_5, K_6^TA_3K_6$ must be nonsingular. But then
 $K_5^TA_3K_6=0$ with $A_3$ nonsingular  is a contradiction. It therefore  holds that $q=r.$ Then
\begin{align*} U^TBU=\texttt{diag}(B_1, 0) , U^TAU=\texttt{diag}(A_1, A_3)
\end{align*}
with $B_1, A_1, A_3$ as desired.

(ii) We note first that $A$ is nonsingular so is $A_3.$ If $A_3 \succ 0,$ then  $A+\mu B \succeq 0$ if and only if  $ A_1+\mu B_1 \succeq 0.$
So it holds in that case  $I_{\succeq}(A,B)=I_{\succeq}(A_1,B_1).$ Otherwise, $A_3$
is either indefinite or negative definite then  $I_{\succeq}(A,B)= \emptyset.$
\end{proof}

The proofs of Theorems \ref{dl1} and \ref{dl2} reveal the following important result.
\begin{cor}\label{cor1}
Suppose $A, B\in\mathcal{S}^n$ are SDC and either $A$ or $B$ is nonsingular.
Then $I_{\succ}(A,B)$ is nonempty if and only if $I_{\succeq}(A,B)$ has more than one point.
\end{cor}

  If $A, B$ are both singular, they can be decomposed to either of the following form.
 \begin{lem}[\cite{Ngan}]\label{lem1}
	For any $A, B\in \mathcal{S}^n,$ there always exists a nonsingular matrix $U$ that
 puts $B$ to
\begin{align*}
		\bar{B}=U^TBU
		=\left(
		\begin{matrix}
			B_1&0_{p\times r}\\ 0_{r\times p}&0_{r\times r} \end{matrix}
		\right)
		\end{align*}
such that $ B_1$ is nonsingular diagonal of size  $p\times p,$ and puts  $A$ to $\bar{A}$
of  either form
\begin{align}\label{a1}
				\bar{A}=U^TAU
		=\left(
		\begin{matrix}
			A_1& A_2\\ A^T_2&0_{r\times r}
		\end{matrix}
		\right)
		\end{align}
or
		\begin{align}\label{a3}	
		\bar{A}=U^TAU
		=\left(
		\begin{matrix}
			A_1&0_{p\times s}& A_2\\ 0_{s\times p}&A_3&0_{s\times (r-s)}\\
		A_2^T&0_{(r-s)\times s}& 0_{(r-s)\times (r-s)}
		\end{matrix}
		\right),
		\end{align}
		where 		$ A_1$ is symmetric of dimension $p\times p,$
		$A_2$ is a $p\times (r-s)$ matrix, and $A_3$ is a   nonsingular diagonal matrix of dimension $s\times s;$
$p, r, s\ge0, p+r=n.$		
	\end{lem}
It is easy to verify that $A, B$ are SDC if and only if $\bar{A}, \bar{B}$ are SDC.
\begin{lem}[\cite{Ngan}]\label{lem2}
\begin{enumerate}
\item[i)]  If  $\bar{A}$  takes the form  \eqref{a1} then $\bar{B}, \bar{A}$  are SDC if and only if  $B_1, A_1$
are SDC and $A_2=0;$
\item[ii)] If $\bar{A}$  take the form  \eqref{a3} then $\bar{B}, \bar{A}$  are SDC if and only if  $B_1, A_1$
are SDC and $A_2=0$ or does not exist, i.e., $s=r.$
\end{enumerate}
\end{lem}
 Now suppose that $A, B$ are SDC, Lemmas \ref{lem1} and \ref{lem2} allow to assume without loss of generality
 that $\bar{B}, \bar{A}$  are already SDC. That is
   \begin{align}\label{pt3}
\bar{B}=U^TBU=\texttt{diag}(B_1, 0),
\bar{A}=U^TAU=\texttt{diag}(A_1, 0)
\end{align}   or
\begin{align}\label{pt4}
\bar{B}=U^TBU=\texttt{diag}(B_1, 0),
\bar{A}=U^TAU=\texttt{diag}(A_1, A_4),
\end{align}
where $A_1, B_1$ are of the same size and diagonal, $B_1$ is nonsingular and
if $\bar{A}$ takes the form \eqref{a1} or \eqref{a3} and $A_2=0$ then
$A_4={\rm diag}(A_3,0)$   or if $\bar{A}$ takes the form \eqref{a3} and  $A_2$ does not exist then $A_4=A_3.$
  Now we can compute $I_{\succeq}(A,B)$ as follows.
\begin{thm}\label{thmb}
\begin{itemize}
\item [(i)] If $\bar{B}, \bar{A}$ take the from $(\ref{pt3}),$ then $I_{\succeq}(A,B)=I_{\succeq}(A_1,B_1);$
\item [(ii)] If $\bar{B}, \bar{A}$ take the from $(\ref{pt4}),$ then $I_{\succeq}(A,B)=I_{\succeq}(A_1,B_1)$
if $A_4 \succeq 0$ and  $I_{\succeq}(A,B)=\emptyset$ otherwise.
\end{itemize}
\end{thm}
We note that $B_1$ is nonsingular, $I_{\succeq}(A_1,B_1)$ is therefore computed by Theorem \ref{dl1}. Especially,
if $I_{\succeq}(A_1,B_1)$ has more than one point, then $I_{\succ}(A_1,B_1)\ne\emptyset,$ see Corollary \ref{cor1}.

In the rest of this section we consider $I_{\succeq}(A,B)$  when  $A, B$ are not SDC.
We need first to show that if $A, B$ are not SDC, then $I_{\succeq}(A,B)$ either is empty or has only one point.
The proof of Lemma \ref{1} is easy, we omit it.
\begin{lem}\label{1}
If $A, B$ are  positive semidefinite then $A, B$ are SDC.
\end{lem}
\begin{lem}\label{dlgiaosu}
If $A, B\in \mathcal{S}^n$ are not SDC then  $I_{\succeq}(A,B)$ either  is empty or has only one element.
\end{lem}
\begin{proof}
Suppose in contrary that $I_{\succeq}(A,B)$ has more than one elements, then we can choose $\mu_1, \mu_2\in I_{\succeq}(A,B), \mu_1 \neq \mu_2$ such that $C:=A+\mu_1B \succeq 0$ and $D:=A+\mu_2B \succeq 0.$
 By Lemma \ref{1}, $C, D$ are SDC, i.e., there is a nonsingular matrix $P$ such that
$P^TCP, P^TDP$ are diagonal. Then $P^TBP$ is diagonal because $P^TCP- P^TDP=(\mu_1-\mu_2)P^TBP$ and $\mu_1 \neq \mu_2.$
Since $P^TAP=P^TCP-\mu_1P^TBP,$ $P^TAP$ is also diagonal. That is $A, B$ are SDC and we get a contradiction.
\end{proof}

To know  when $I_{\succeq}(A,B)$ is empty or has one element, we  need the following result.
\begin{lem}[\cite{Uhlig76}]\label{dlUhlig1}
Let $A, B\in   \mathcal{S}^n,$  B be nonsingular. Let
$B^{-1}A$  have the real Jordan normal form $diag(J_1,\ldots J_r,J_{r+1},\ldots, J_m)$, where $J_1, \ldots, J_r$ are Jordan blocks  corresponding to real eigenvalues $\lambda_1, \lambda_2,\ldots, \lambda_r$ of $B^{-1}A$ and $J_{r+1},\ldots, J_m$ are Jordan blocks  for pairs of complex conjugate roots $\lambda_i=a_i \pm {\bf i}b_i, a_i, b_i \in \Bbb R, i=r+1, r+2, \ldots, m$
of $B^{-1}A$. Then there exists a nonsingular matrix $U$ such that
\begin{align}\label{mtB}
 U^TBU=\texttt{diag}(\epsilon_1E_1,\epsilon_2E_2,\ldots,\epsilon_rE_r,E_{r+1},\ldots, E_m)
\end{align}
\begin{align}\label{mtA}
 U^TAU=\texttt{diag}(\epsilon_1E_1J_1,\epsilon_2E_2J_2,\ldots,\epsilon_rE_rJ_r,E_{r+1}J_{r+1},\ldots, E_mJ_m)
\end{align}
where $\epsilon_i=\pm1, E_i=\left(\begin{matrix}0&0&\ldots&0&1\\
 0&0&\ldots&1&0\\
\ldots&\ldots&\ldots&\ldots&\ldots\\
\ldots&\ldots&\ldots&\ldots&\ldots\\
1&0&\ldots&0&0
\end{matrix}\right);$ ${\rm dim} E_i={\rm dim} J_i=n_i; n_1+n_2+\ldots+n_m=n.$
\end{lem}

\begin{thm}\label{dl3}
Let $A, B\in \mathcal{S}^n$ be as in Lemma \ref{dlUhlig1} and $A, B$ are not SDC. The followings hold.
\begin{enumerate}
\item [(i)] if $A \succeq 0$ then $I_{\succeq}(A,B)=\{0\};$
\item [(ii)] if  $A\nsucceq 0$ and there is a real eigenvalue $\lambda_l$ of $B^{-1}A$ such that $A+(-\lambda_l)B \succeq 0$ then $$I_{\succeq}(A,B)=\{-\lambda_l\};$$
\item [(iii)] if (i) and (ii) do not occur then $I_{\succeq}(A,B)=\emptyset.$
\end{enumerate}

\end{thm}
\begin{proof}
It is sufficient to prove only (iii).
Lemma \ref{dlUhlig1} allows us to  decompose $A$ and $B$ to the forms \eqref{mtA} and \eqref{mtB}, respectively. Since $A, B$ are not SDC,
at least one of the following cases must occur.

{\bf Case 1} {\em There is a Jordan block $J_i$ such that  $n_i\geq 2$ and $\lambda_i \in \Bbb R.$ }
We then consider the following principal minor of $A+\mu B:$
$$Y=\epsilon_i(E_iJ_i+\mu E_i)=\epsilon_i \left( \begin{matrix}0&0&\ldots&0&\lambda_i+\mu\\
 0&0&\ldots&\lambda_i+\mu&1\\
\ldots&\ldots&\ldots&\ldots&\ldots\\
\ldots&\ldots&\ldots&\ldots&\ldots\\
\lambda_i+\mu&1&\ldots&0&0 \end{matrix}\right)_{n_i\times n_i}.$$
If $n_i =2$ then $Y=\epsilon_i\left( \begin{matrix}0&\lambda_i+\mu\\
 \lambda_i+\mu&1 \end{matrix}\right).$
 Since  $ \mu\ne -\lambda_i,$ $Y\not\succeq0$ so  $A+\mu B\not\succeq0.$
If $n_i>2$ then $Y$ always contains the following not positive semidefinite principal minor of size $(n_i-1)\times(n_i-1):$
$$\epsilon_i \left( \begin{matrix}
 0&0&\ldots&\lambda_i+\mu&1\\
0&0&\ldots&1&0\\
\ldots&\ldots&\ldots&\ldots&\ldots\\
\lambda_i+\mu&1&\ldots&0&0\\
1&0&\ldots&0&0 \end{matrix}\right)_{(n_i-1)\times(n_i-1)}.$$    So  $A+\mu B\not\succeq0.$

{\bf Case 2} {\em   There is a  Jordan block $J_i$ such that $n_i\geq 4$ and $\lambda_i=a_i\pm {\bf i}b_i \notin \Bbb R.$}
We then consider
$$Y=\epsilon_i(E_iJ_i+\mu E_i)=\epsilon_i \left( \begin{matrix}0&0&\ldots&b_i&a_i+\mu\\
 0&0&\ldots&a_i+\mu&-b_i\\
\ldots&\ldots&\ldots&\ldots&\ldots\\
b_i&a_i+\mu&\ldots&0&0\\
a_i+\mu&-b_i&\ldots&0&0 \end{matrix}\right)_{n_i\times n_i}.$$
This matrix always contains either a principal minor of size $2 \times 2:$
$\epsilon_i \left( \begin{matrix}b_i&a_i+\mu\\
a_i+\mu&-b_i \end{matrix}\right)$
 or  a principal minor of size $4 \times 4:$
$$ \epsilon_i \left( \begin{matrix}0&0&b_i&a_i+\mu\\
 0&0&a_i+\mu&-b_i\\
b_i&a_i+\mu&0&0\\
a_i+\mu&-b_i&0&0 \end{matrix}\right).$$
 Both are not positive semidefinite  for any $\mu\in\Bbb R.$
\end{proof}

Similarly, we have the following result. We omit its proof.
\begin{thm}\label{dl4}
Let $A, B\in \mathcal{S}^n$ be  not SDC. Suppose $A$ is nonsingular and $A^{-1}B$ has real Jordan normal form $diag(J_1,\ldots J_r,J_{r+1},\ldots, J_m)$, where $J_1, \ldots, J_r$ are Jordan blocks  corresponding to real eigenvalues $\lambda_1, \lambda_2,\ldots, \lambda_r$ of $A^{-1}B$ and $J_{r+1},\ldots, J_m$ are Jordan blocks  for pairs of complex conjugate roots $\lambda_i=a_i \pm {\bf i}b_i, a_i, b_i \in \Bbb R, i=r+1, r+2, \ldots, m$
of $A^{-1}B$.
\begin{enumerate}
\item [(i)] If $A \succeq 0$ then $I_{\succeq}(A,B)=\{0\};$
\item [(ii)] If  $A\nsucceq0$ and there is a real eigenvalue $\lambda_l \neq 0$  of $A^{-1}B$
  such that $A+\left(-\dfrac{1}{\lambda_l}\right)B \succeq 0$ then $I_{\succeq}(A,B)=\left\{-\dfrac{1}{\lambda_l}\right\};$
\item [(iii)] If cases $(i)$ and $(ii)$ do not occur then $I_{\succeq}(A,B)=\emptyset.$
\end{enumerate}
\end{thm}

Finally, if  $A$ and $B$ are not SDC and  both singular.   Lemma \ref{lem1} indicates that $A$ and $B$ can be simultaneously
decomposed to $\bar{A}$ and $\bar{B}$  in either  \eqref{a1} or \eqref{a3}.
If $\bar{A}$ and $\bar{B}$  take the forms  \eqref{a1} and  $A_2=0$
then $I_{\succeq}(A,B)=I_{\succeq}(A_1,B_1),$ where $A_1, B_1$ are not SDC and $B_1$ is nonsingular.
In this case we apply Theorem \ref{dl3}
to compute $I_{\succeq}(A_1,B_1).$ If $\bar{A}$ and $\bar{B}$  take the forms  \eqref{a3} and  $A_2=0.$
In this case, if $A_3$ is not positive definite then  $I_{\succeq}(A,B)=\emptyset.$ Otherwise,
$I_{\succeq}(A,B)=I_{\succeq}(A_1,B_1),$ where $A_1, B_1$ are not SDC and $B_1$ is nonsingular, again we can apply
Theorem \ref{dl3}. Therefore we need only to consider
 the case $A_2\ne0$ with noting  that  $I_{\succeq}(A,B)\subset I_{\succeq}(A_1,B_1).$

\begin{thm}\label{thm5}
Given $A, B\in \mathcal{S}^n $ are not SDC and singular such that
   $\bar{A}$ and $\bar{B}$  take the forms in either  \eqref{a1} or \eqref{a3} with $A_2\ne 0.$ Suppose that $I_{\succeq}(A_1,B_1)=[a,b], a<b.$
Then, if $a\not\in I_{\succeq}(A, B)$ and  $b\not\in I_{\succeq}(A, B)$ then  $I_{\succeq}(A, B)=\emptyset.$
 \end{thm}
\begin{proof} We consider $\bar{A}$ and $\bar{B}$   in  \eqref{a3}, the form in
\eqref{a1} is considered similarly. Suppose in contrary that
   $I_{\succeq}(A, B)=\{\mu_0\}$ and $a<\mu_0<b.$
Since $I_{\succeq}(A_1,B_1)$ has more than one point, by Lemma \ref{dlgiaosu}, $A_1, B_1$ are SDC.
Let $Q_1$ be a $p\times p$ nonsingular matrix such that $Q_1^TA_1Q_1, Q_1^TB_1Q_1$ are diagonal, then
 $Q^T_1(A_1+\mu_0  B_1)Q_1:=\rm{diag}(\gamma_1,\gamma_2,\ldots,\gamma_p)$ is
 a diagonal matrix. Moreover,  $B_1$ is nonsingular, we have
$I_{\succ}(A_1,B_1)=(a,b),$ please see Corollary \ref{cor1}. Then
  $\gamma_i>0 $ for $i=1,2,\ldots, p$ because $\mu_0\in I_{\succ}(A_1,B_1).$
 Let  $Q:=\rm{diag}(Q_1,I_s,I_{r-s})$ we  then have
\begin{align*} 	Q^T(\bar{A}+\mu_0 \bar{ B})Q= \left(
		\begin{matrix}
			Q^T_1(A_1+\mu_0 B_1)Q_1&0_{p\times s}&Q^T_1 A_2\\ 0_{s\times p}&A_3&0_{s\times (r-s)}\\
		A_2^TQ_1&0_{(r-s)\times s}& 0_{(r-s)\times (r-s)}
		\end{matrix}
		\right).
\end{align*}
We note that $I_{\succeq}(A, B)=\{\mu_0\}$ is singleton implying   ${\rm det}(A+\mu_0B)=0$
and so  ${\rm det}(Q^T(\bar{A}+\mu_0 \bar{ B})Q)=0.$ On the other hand, since
$A_3$ is nonsingular diagonal and $A_1+\mu_0B_1\succ0,$ the first $p+s$ columns of the matrix
$Q^T(\bar{A}+\mu_0 \bar{ B})Q$ are linearly independent. One of the following cases must occur:
i) the columns of  the right side submatrix  $\left(
		\begin{matrix}Q^T_1A_2\\0_{s\times (r-s)}\\0_{(r-s)\times (r-s)} \end{matrix}
		\right)$ are linearly independent and  at least one column, suppose $(c_1, c_2, \ldots,  c_p, 0, 0,  \ldots, 0 )^T,$  is a
linear combination of the columns of the matrix
$$\left(
		\begin{matrix}Q^T_1(A_1+\mu_0. B_1)Q_1\\0_{s\times p}\\A_2^TQ_1 \end{matrix}
		\right):=({\rm column_1}|{\rm column_2}|\ldots |{\rm column_p}),$$
where ${\rm column_i}$ is the $i$th column of the matrix  or ii) 
the columns of the right side submatrix  $\left(
		\begin{matrix}Q^T_1A_2\\0_{s\times (r-s)}\\0_{(r-s)\times (r-s)} \end{matrix}
		\right)$ are linearly dependent.
If the case i) occurs then  there are scalars
$a_1, a_2, \ldots, a_p$ which are not all zero such that
\begin{align}\label{B2}
\left(
		\begin{matrix}c_1\\c_2\\ \vdots \\ c_p\\ 0 \\  \vdots \\0 \end{matrix}
		\right)=a_1 {\rm column_1}+a_2{\rm column_2}+\ldots+a_p{\rm column_p}.
\end{align}

 Equation \eqref{B2} implies that
$$ \begin{cases} c_1=a_1 \gamma_1\\ c_2=a_2  \gamma_2\\ \ldots \\ c_p = a_p  \gamma_p\\ 0=a_1 c_1+a_2 c_2+\ldots+a_p c_p\end{cases}$$
which further implies
$$ 0= (a_1)^2 \gamma_1+(a_1)^2  \gamma_2+\ldots+(a_p)^2  \gamma_p.$$ This cannot happen with $\gamma_i>0$ and $(a_1)^2+(a_2)^2+\ldots+(a_p)^2 \neq 0.$ This contradiction shows that
$I_{\succeq}(A,B)=\emptyset.$
If the case ii) happens  then there always exists a nonsingular matrix $H$ such that 
\begin{align*} 	H^TQ^T(\bar{A}+\mu_0 \bar{ B})QH= \left(
		\begin{matrix}
			Q^T_1(A_1+\mu_0 B_1)Q_1&0_{p\times s}&\hat{A}_2&0 
\\ 0_{s\times p}&A_3&0 &0\\
\hat{A}_2^T&0&0&0\\
		0&0 & 0 &0
		\end{matrix}
		\right),
\end{align*}
where $\hat{A}_2$ is a full column-rank matrix. Let
\begin{align*} \hat{A}= \left(\begin{matrix}
			Q^T_1A_1Q_1&0_{p\times s}&\hat{A}_2\\ 
0_{s\times p}&A_3&0\\
\hat{A}_2^T&0&0
		\end{matrix}
		\right), \hat{B}= \left(\begin{matrix}
			Q^T_1B_1 Q_1&0_{p\times s}&0\\ 
0_{s\times p}&0&0 \\
0&0&0 
		\end{matrix}
		\right),
\end{align*}
we have $I_{\succeq}(A,B)=I_{\succeq}(\bar{A},\bar{B})=I_{\succeq}(\hat{A},\hat{B})$
and so $I_{\succeq}(\hat{A},\hat{B})=\{\mu_0\}.$ This implies
${\rm det}(\hat{A}+\mu_0\hat{B})=0,$ and the right side submatrix $\left(
		\begin{matrix}\hat{A}_2\\0\\0\end{matrix}
		\right)$  is full column-rank.
 We  return to the case i). 
\end{proof}

 \section{Application for the GTRS}\label{sec2}
 In this section we find an optimal Lagrangian multiplier
 $$\mu^*\in I:= I_{\succeq}(A,B)\cap [0,\infty)$$ together with an optimal solution $x^*$  for  the GTRS \eqref{original}.
 We need first to recall the following optimality conditions for the GTRS \eqref{original}.
   \begin{lem}[\cite{MJ}]\label{More}
 A vector $x^*\in\Bbb R^n$ is an optimal solution to GTRS \eqref{original}  if and only if there exists  $\mu^*\ge0$ such that
\begin{align}
(A+\mu^*B)x^*+a+\mu^* b=0,\label{dk1}\\
g(x^*)\le0,\label{dk2}\\
\mu^* g(x^*)=0, \label{dk3}\\
A+\mu^* B\succeq0. \label{dk4}
\end{align}
\end{lem}
In fact, the conditions \eqref{dk1} and \eqref{dk4} are necessary and sufficient   for the GTRS \eqref{original}
to be bounded below  \cite{Xia}. However,
 a bounded GTRS may have no optimal solution, see, for example \cite{Adachi}. The conditions \eqref{dk2}-\eqref{dk3}
 are thus added to guarantee the existence of an optimal solution to GTRS \eqref{original}. To check
  whether a $\mu \in I$ satisfies  conditions \eqref{dk1}-\eqref{dk3} we  need to apply
the following result.
 \begin{lem}[\cite{MJ}]\label{More1}
 Suppose $I_{PD}:=I_{\succ}(A,B)\cap [0,\infty)$ is nonempty. If $x(\mu)$ is the solution of
 \eqref{dk1}, and if the function $\varphi:\Bbb R \rightarrow \Bbb R$
 is defined on $I_{PD}$ by
 $$\varphi(\mu)=g[x(\mu)],$$
 then  $\varphi$ is strictly decreasing on $I_{PD}$ unless $x( \cdot)$ is constant on $I_{PD}.$
\end{lem}
  We note that $I$ can be a nonempty interval while $I_{PD}$ is empty. Fortunately,
  our  Corollary \ref{cor1} shows that this case happens only when  $A$ and $B$  are both singular.
  Then we can decompose $A, B$ to $\bar{A}, \bar{B}$ and apply
  Theorem \ref{thmb} if $A, B$ are SDC,  Theorem \ref{dl3} or Theorem \ref{dl4}
  or Theorem \ref{thm5} if $A, B$ are not SDC.

By Lemma \ref{More},   if
   $I=\emptyset,$ the GTRS \eqref{original}
  has no optimal solution,  it is even unbounded from below \cite{Xia}.
  If $I\ne\emptyset$ but has only one element:  $I=I_{\succeq}(A,B)\cap [0,+\infty)=\{\mu\}.$ Then
  we need only to solve the linear system  \eqref{dk1}-\eqref{dk3}
  to check whether $\mu^*=\mu.$ If $I$ is an interval, suppose
   $I=[\mu_1, \mu_2],$ where $\mu_1\ge0$ and $\mu_2$ may be $\infty,$ then $A, B$ are SDC.
We assume that $A={\rm diag}(\alpha_1,\alpha_2,\ldots,\alpha_n),
B={\rm diag}(\beta_1,\beta_2,\ldots,\beta_n).$
The equation  \eqref{dk1} is then
 of the following simple form
\begin{align}\label{equations}
(\alpha_i+\mu \beta_i) x_i= -(a_i+\mu b_i), i=1,2,\ldots,n.
\end{align}
Solving \eqref{equations} for a fixed $\mu$ is very simple, the main duty is thus
finding $\mu^*$ in the following cases.
\begin{enumerate}
\item If at least one of the matrices $A$ and $B$ is nonsingular then $I_{PD}\ne\emptyset$ and
 $I ={\rm closure}(I_{PD}).$
The GTRS \eqref{original} then attains a unique optimal solution $x^*$ at an optimal Lagrange multiplier $\mu^*$ \cite{MJ}.
We first check whether $\mu^*=0.$ If not, we apply Lemma \ref{More1} for finding $\mu^*$ such that $\varphi(\mu^*)=0.$
Observe from \eqref{equations} that $x(\mu)$ is constant on $I_{PD}$
only when $\beta_i=b_i=0$ for all $i=1,2,\ldots,n.$ This case can be dealt with easily since $g(x)$ is then constant.
 Otherwise, $\varphi(\mu)$ is  strictly decreasing on $I_{PD}$
and we have the following results.
\begin{lem}[\cite{Feng},\cite{Adachi}]\label{findmu}
Suppose the Slater condition holds for the GTRS \eqref{original}, i.e., there exists $\bar{x}\in\Bbb R^n$ such that
$g(\bar{x})<0,$ and $I_{PD}\ne\emptyset.$
\begin{enumerate}
\item[(a)] If $\varphi(\mu)>0$ on $I_{PD}$ and $\mu_2<\infty,$ then $\mu^*=\mu_2;$
\item[(b)] If $\varphi(\mu)<0$ on $I_{PD}$  then $\mu^*=\mu_1;$
\item[(c)] If $\varphi(\mu)$ changes its sign on $I_{PD}$ then $\mu^*\in  I_{PD};$
\item[(d)] If $\varphi(\mu)>0$ on $I_{PD}$ then $\mu_2$ cannot be $\infty.$
\end{enumerate}
\end{lem}
The case (d) indicates that if $I=[\mu_1, \infty)$  and $\varphi(\mu_1)>0$ then   $\mu_1<\mu^*<\infty.$
\item If both $A$ and $B$ are singular,   by Lemma \ref{lem2}, $B, A$ are decomposed to
the form either \eqref{pt3} or \eqref{pt4}, and are now in the form 
\begin{align}\label{pt5}
B=\texttt{diag}(\beta_1,  \ldots,\beta_p,0,\ldots,0),
A=\texttt{diag}(\alpha_1,  \ldots,\alpha_p,0,\ldots,0)
\end{align}   or
\begin{align}\label{pt6}
B=\texttt{diag}(\beta_1,  \ldots,\beta_p,0,\ldots,0),
A=\texttt{diag}(\alpha_1,  \ldots,\alpha_p,\alpha_{p+1},\ldots,  \ldots,\alpha_{n})
\end{align}
where $\beta_1, \beta_2,\ldots,\beta_p$  are nonzero. We  have
$$I_{\succeq}(A,B)=I_{\succeq}(A_1,B_1)={\rm closure}\left(I_{\succ}(A_1,B_1)\right),$$
where $B_1=\texttt{diag}(\beta_1, \beta_2,\ldots,\beta_p), A_1=\texttt{diag}(\alpha_1, \alpha_2,\ldots,\alpha_p),$
please see  Theorem \ref{thmb} and Corollary \ref{cor1}.  We note also that
if $\alpha_i>0$ for $i=p+1, \ldots, n,$ then $I_{\succ}(A,B)=I_{\succ}(A_1, B_1)$
and we can apply Lemma \ref{findmu}. Otherwise, $I_{\succ}(A,B)=\emptyset$
and the GTRS \eqref{original} may have no optimal solution. We deal with this case as follows.

If  $A, B$ take the form \eqref{pt5},
the equations \eqref{equations} become
\begin{align}\label{equationsp}
(\alpha_i+\mu \beta_i) x_i = -(a_i+\mu b_i), i=1,2,\ldots,p;\\
0 = -(a_i+\mu b_i), i=p+1,\ldots,n.\nonumber
\end{align}
If $a_i=b_i=0$ for $i=p+1,\ldots,n,$ then \eqref{original} is reduced to a GTRS of $p$ variables
with matrices  $A_1, B_1$ such that $I_{\succ}(A_1,B_1)\ne\emptyset.$
We then apply Lemma \ref{findmu} for it. Otherwise,
either \eqref{equationsp} has no solution $x$ for all $\mu\in I$  or
it has solutions at only one  $\mu\in I.$  Then we  check easily whether $\mu^*=\mu.$

If $A, B$ take  the form \eqref{pt6}, the equations \eqref{equations} become
\begin{align}\label{equationsp1}
(\alpha_i+\mu \beta_i) x_i &= -(a_i+\mu b_i), i=1,2,\ldots,p;\\
\alpha_i x_i &= -(a_i+\mu b_i), i=p+1,p+2,\ldots,p+s;\nonumber\\
0 &= -(a_i+\mu b_i), i=p+s+1,\ldots,n.\nonumber
\end{align}
By the same arguments as above, either \eqref{original} is then reduced to a GTRS of $p+s$ variables
with matrices
\begin{align*}
&\bar{A}_1=\texttt{diag}(A_1, \alpha_{p+1},  \ldots,\alpha_{p+s})=\texttt{diag}(\alpha_1,\ldots, \alpha_p, \alpha_{p+1}, \ldots,\alpha_{p+s}),\\
&\bar{B}_1=\texttt{diag}(B_1, \underbrace{0,\ldots, 0}_{s \text{ zeros }})=\texttt{diag}(\beta_1,  \ldots,\beta_p,\underbrace{0,\ldots, 0}_{s \text{ zeros }})
\end{align*}
  such that $I_{\succ}(\bar{A}_1,\bar{B}_1)\ne\emptyset.$
We then apply Lemma \ref{findmu} for it. Otherwise,
either \eqref{equationsp1} has no solution $x$ for all $\mu\in I$  or
it has solutions at only one  $\mu\in I.$  

\end{enumerate}

\section{Conclusion and remarks}
In this paper, we showed that for a given pair of real symmetric matrices $A, B,$ the set $I_{\succeq}(A, B)$ of real values  $\mu\in\Bbb R$ such that the matrix pencil $A+\mu B$ is positive semidefinite is always computable by solving the generalized eigenvalue problem
 of   $n\times n$ dimension.
The computation is considered in two separated cases:
 If   $A, B$ are not simultaneously diagonalizable via congruence (SDC),  $I_{\succeq}(A, B)$
is either  empty or singleton, while if $A, B$ are SDC, $I_{\succeq}(A, B)$ can be empty, singleton or an interval. 
In case $I_{\succeq}(A, B)$
 is an interval, if  either $A$ or $B$ is nonsingular,   $I_{\succeq}(A, B)$ is the closure of 
  the positive definite interval $I_{\succ}(A, B).$ Otherwise, $A, B$ are decomposable to block diagonals
 of submatrices $A_1, B_1$  with $B_1$ nonsingular such that
  $I_{\succeq}(A, B)$ is now the closure of $I_{\succ}(A_1, B_1).$
  With $I_{\succeq}(A, B)$ in hand, we are able to solve the  generalized trust region subproblem \eqref{original}
 not only in the {\em easy-case} when $I_{\succ}(A, B)$ is nonempty but also in the {\em hard-case} by only solving the linear
 equations. 
Our result needs only to solve the  generalized  eigenvalue problem of   $n\times n$  dimension compared with
an $(2n + 1)\times (2n + 1)$ generalized eigenvalue problem in \cite{Adachi}. On the other hand,
we can completely deal with the hard-case of the GTRS \eqref{original}, which was an open problem in
\cite{MJ} and \cite{Hsia}.

\section*{References}

\bibliography{mybibfile}

\end{document}